\documentclass[12pt]{article} 

\usepackage{amsmath}
\usepackage{amsthm}
\usepackage{amsfonts}
\usepackage{mathrsfs}
\usepackage{stmaryrd}
\usepackage{setspace}
\usepackage{fullpage}
\usepackage{amssymb}
\usepackage{breqn}
\usepackage{enumitem}
\usepackage{bbold} 
\usepackage{authblk}
\usepackage{comment}
\usepackage{hyperref}
\usepackage{pgf,tikz}
\usepackage{graphicx}
\usepackage{subcaption}

\newtheorem{thm}{Theorem}[section]

\newtheorem{lemma}[thm]{Lemma}

\newtheorem{prop}[thm]{Proposition}

\newtheorem{clm}[thm]{Claim}

\newcommand\ex{\ensuremath{\mathrm{ex}}}

\newcommand\cN{{\mathcal N}}

\newcommand{\ignore}[1]{}

\title{Generalized Turán problems for double stars}
\author{Dániel Gerbner}
\date{\small Alfr\'ed R\'enyi Institute of Mathematics}

\begin{document}

\maketitle

\begin{abstract}
    We study the generalized Tur\'an function $\ex(n,H,F)$, when $H$ or $F$ is a double star $S_{a,b}$, which is a tree with a central edge $uv$, $a$ leaves connected to $u$ and $b$ leaves connected to $v$. We determine $\ex(n,K_k,S_{a,b})$ and $\ex(n,S_{a,b},F)$ for sufficiently large $n$, where $F$ is either a 3-chromatic graph with an edge whose deletion results in a bipartite graph, or the 2-fan, i.e. two triangles sharing a vertex. We also give bounds on $\ex(n,S_{a,b},S_{c,d})$.
\end{abstract}

\section{Introduction}

Ordinary Tur\'an problems deal with the quantity $\ex(n,F)$, which is the largest number of edges in $n$-vertex graphs that do not contain $F$ as a subgraph. Tur\'an' theorem \cite{T} states that $\ex(n,K_{k+1})$ is equal to the number of edges in the $k$-partite Tur\'an graph, which is the complete $k$-partite graph with each part having order either $\lfloor n/k\rfloor$ or $\lceil n/k\rceil$.

Generalized Tur\'an problems deal with $\ex(n,H,F)$, which is the largest number of copies of $H$ in $n$-vertex graphs that do not contain $F$ as a subgraph. After several sporadic results, the systematic study of generalized Tur\'an problems was initiated by Alon and Shikhelman \cite{ALS2016}.

Gy\H ori, Wang and Woolfson \cite{gyww} initiated the study of generalized Tur\'an problems involving double stars. The double star $S_{a,b}$ consists of a \textit{central edge} $uv$, $a$ leaves connected to $u$ and $b$ leaves connected to $v$. We say that $u$ and $v$ are \textit{centers} of the double star.
Gy\H ori, Wang and Woolfson \cite{gyww} showed the following.

\begin{thm}[Gy\H ori, Wang and Woolfson \cite{gyww}]\label{GyWW} If $n$ is sufficiently large, then $\ex(n,S_{a,b},K_3)=\cN(S_{a,b},K_{m,n-m})$ for some $m$.
\end{thm}

The above theorem states that among triangle-free graphs, the number of copies of a double star is always maximized by a complete bipartite graph. Given $a$ and $b$, it is straightforward to determine which complete bipartite graph contains the most copies of $S_{a,b}$ (i.e. to determine $m$), but we cannot handle this problem in such generality.

Similar situation, when $\ex(n,H,F)$ is almost determined for some classes of graphs, only some optimization of polynomials is missing, has already occurred in generalized Tur\'an problems. Gerbner and Palmer \cite{gerpal} introduced the $F$-Tur\'an-good graphs, which are graphs $H$ such that $\ex(n,H,F)$ is equal to the number of copies of $H$ in the $(\chi(F)-1)$-partite Tur\'an graph for $n$ sufficiently large. We say that $H$ is \textit{weakly $F$-Tur\'an-good} if for $n$ sufficiently large we have $\ex(n,H,F)=\cN(H,G)$ for some complete $(\chi(F)-1)$-partite $n$-vertex graph. We can reformulate Theorem \ref{GyWW} the following way: $S_{a,b}$ is weakly $K_3$-Tur\'an-good. 

Gy\H ori, Pach and Simonovits \cite{gypl} showed that complete multipartite graphs are weakly $K_k$-Tur\'an-good for any $k$.
Gerbner \cite{ger3} showed that if $F$ is a 3-chromatic graph with a color-critical edge (which is an edge whose deletion decreases the chromatic number), then complete multipartite graphs are weakly $F$-Tur\'an-good.

It is easy to see that counting copies of $S_{a,b}$ or $K_{a+1,b+1}$ in complete bipartite graphs is the same optimization problem. Indeed, every copy of $K_{a+1,b+1}$ contains exactly $(a+1)(b+1)$ copies of $S_{a,b}$, and every copy of $S_{a,b}$ is counted exactly once. In fact, Theorem \ref{GyWW} implies the above mentioned result of Gy\H ori, Pach and Simonovits \cite{gypl} for $k =3$, as $\cN(K_{a+1,b+1},G)\le (a+1)(b+1)\cN(S_{a,b},G)\le (a+1)(b+1)\cN(S_{a,b},K_{m,n-m})=\cN(K_{a+1,b+1},K_{m,n-m})$. 
More generally, we have the following.

\begin{prop}\label{propi}
Assume that $H$ has a unique proper $r$-coloring, $F$ has chromatic number $r+1$ and $H$ is weakly $F$-Tur\'an-good. If $H'$ is obtained by adding edges to $H$ and has chromatic number $r$, 
then $H'$ is also weakly $F$-Tur\'an-good. Moreover, the same complete $r$-partite graph contains the maximum number of copies of $H$ and $H'$.
\end{prop}

\begin{proof}
Let us assume that there are $p$ ways to obtain $H'$ from $H$ by adding edges, and $H'$ contains $q$ copies of $H$. Observe that if a copy of $H$ is in a complete $r$-partite graph $G_0$, then all the $p$ ways to obtain $H'$ create a subgraph of $G_0$. Let $n$ be large enough, and let $G_0$ be a complete $r$-partite $n$-vertex graph with $\cN(H,G_0)=\ex(n,H,F)$. Then for any $F$-free $n$-vertex graph $G$ we have $\cN(H',G)\le p\cN(H,G)/q\le p\ex(n,H,F)/q=p\cN(H,G_0)/q=\cN(H',G_0)$, completing the proof.
\end{proof}

In this paper we examine further generalized Tur\'an problems where the forbidden graph or the graph we count or both are double stars. 
Our first result extends Theorem \ref{GyWW} to 3-chromatic graphs with a color-critical edge in place of the triangle.

\begin{thm}\label{cce}
If $F$ has chromatic number three and a color-critical edge, then $S_{a,b}$ is weakly $F$-Tur\'an-good.
\end{thm}

We remark that the asymptotics has been known, as $\ex(n,S_{a,b},F)=(1+o(1))\ex(n,S_{a,b},K_3)$ for any 3-chromatic $F$ by a theorem of Gerbner and Palmer \cite{GP}.
Also note that using Proposition \ref{propi}, Theorem \ref{cce} implies the above mentioned result of Gerbner \cite{ger3}, i.e. that complete bipartite graphs are weakly $F$-Tur\'an-good.

The smallest example of a 3-chromatic graph with no color-critical edge is the 2-fan $F_2$, which consists of two triangles sharing a vertex. It is easy to see that to a complete bipartite graph with at least two vertices in both classes, we can add exactly one edge without creating a copy of $F_2$. Erd\H os, F\"uredi, Gould and Gunderson \cite{efgg} showed that an $F_2$-free graph has at most $\lfloor n^2/4\rfloor+1$ edges, one more than the Tur\'an graph. Gerbner and Palmer \cite{GP} showed that $C_4$ is $F_2$-Tur\'an-good (the single additional edge cannot be in any copy of $C_4$). Gerbner \cite{ger2} extended this to any $K_{a,a}$ with $a>2$ in place of $C_4$. Let $K_{a,b}^+$ denote the graph we obtain from $K_{a,b}$ by adding an edge to the part of size $a$. Gerbner and Patk\'os \cite{gerpat} showed that $\ex(n,K_{a,b},F_2)=\cN(H,K_{m,n-m}^+)$ for some $m$.

\begin{thm}\label{fketto}
For any positive integers $a,b$ and sufficiently large $n$ we have $\ex(n,S_{a,b},F_2)=\cN(S_{a,b},K_{m,n-m}^+)$ for some $m$.
\end{thm}

Perhaps the most natural questions concerning generalized Tur\'an problems for double stars are counting cliques or counting double stars when a double star is forbidden. We resolve both questions by the following results.

\begin{prop}\label{klikks}
Let $n=p(a+b+1)+q$ with $q\le a+b$. Then for $k\ge 3$ we have $\ex(n,K_k,S_{a,b})=p\binom{a+b+1}{k}+\binom{q}{k}$.
\end{prop}

Concerning $\ex(n,S_{a,b},S_{c,d})$, we assume without loss of generality that $a\le b$ and $c\le d$, and separate two cases. If $c\le a$, then we can assume $b<d$, as otherwise $\ex(n,S_{a,b},S_{c,d})=0$.

Let $r=\max\{(c+d)\binom{c+d-1}{a}\binom{c+d-1-a}{b},d\binom{d-1}{a}\binom{d-1}{b}\}$ if $a\neq b$, and \\ $r=\max\{\frac{c+d}{2}\binom{c+d-1}{a}\binom{c+d-1-a}{b},\frac{d}{2}\binom{d-1}{a}\binom{d-1}{b}\}$ if $a=b$. We say that $r$ is \textit{nice} if $r=d\binom{d-1}{a}\binom{d-1}{b}$ for $a\neq b$ or $r=\frac{d}{2}\binom{d-1}{a}\binom{d-1}{b}$ for $a=b$.

\begin{prop}\label{neww}
Let $c\le a\le b<d$. Then there is a constant $k=k(a,b,c,d)$ such that for any $n$, we have $rn-k\le \ex(n,S_{a,b},S_{c,d})\le rn$. Moreover, if $r$ is nice, $dn$ is even and $n$ is large enough, then $\ex(n,S_{a,b},S_{c,d})= rn$. If $r$ is not nice and $c+d+1$ divides $n$, then $\ex(n,S_{a,b},S_{c,d})= rn$.
\end{prop}

\begin{thm}\label{cnc}
Let $a\le b$ and $a<c\le d$. If $n$ is sufficiently large, then $\ex(n,S_{a,b},S_{c,d})=\cN(S_{a,b},K_{c,n-c})$.
\end{thm}

The rest of the paper is structured as follows. In Section 2 we state the well-known results we use, and afterwards we state and prove the lemmas needed to prove Theorems \ref{cce} and \ref{fketto}. Section 3 contains the proofs of the theorems stated above.

\section{Preliminaries}

We will use a theorem of Abbott, Hanson and Sauer \cite{ahs}. Let $f(n)$ be equal to $n(n-1)$ if $n$ is even and to $n^2+(n-1)/2$ if $n$ is odd.

\begin{thm}[Abbott, Hanson and Sauer \cite{ahs}]\label{abhasa}
If a graph has $f(n)$ edges, then it contains a star or a matching with $n$ edges.
\end{thm}

Note that $f(n)$ is the best possible bound, but we will only use the weaker result that any graph $G$ contains a star or matching with at least $\sqrt{|E(G)|}$ edges.

We will also use some other well-known results. The K\H ov\'ari-S\'os-Tur\'an theorem \cite{kst} states that $\ex(n,K_{s,t})=O(n^{2-1/s})$. The Erd\H os-Simonovits stability theorem \cite{erd1,erd2,sim} states that for any graph $F$, if $G$ is an $F$-free graph on $n$ vertices with $ \ex(n,F)-o(n^2)$ edges, then we can obtain the $(\chi(f)-1)$-partite Tur\'an graph from $G$ by adding and deleting $o(n^2)$ edges.

$B_t$ denotes the book graph with $t$ pages, i.e. $t$ triangles sharing an edge. Alon and Shikelman \cite{ALS2016} showed $\ex(n,K_3,B_t)=o(n^2)$ for any fixed $t$. In fact, they showed more: for any $\varepsilon>0$ there exists $\delta>0$ such that if an $n$-vertex graph $G$ contains no $B_{\delta n}$, then $G$ contains as most $\epsilon n^2$ triangles.

We continue with the lemmas needed for the proof of Theorem \ref{cce}. 

\begin{lemma}\label{lem1}
Let $t$ be a positive integer, $\beta<1$ and $n$ large enough. Let $G$ be an $n$-vertex $B_t$-free graph with maximum degree $n/2<\Delta<\beta n$. Then $|E(G)|\le\Delta(n-\Delta)$. Moreover, if $G$ contains a triangle, then $|E(G)|\le\Delta(n-\Delta)-cn$ for some $c$ depending on $\beta$ and $t$.
\end{lemma}

\begin{proof}
Assume first that $G$ is triangle-free. Let $v$ be a vertex of $G$ with degree $\Delta$, and $U$ be the set of neighbors of $v$. Then $U$ is an independent set, and the $n-\Delta$ vertices in $V(G)\setminus U$ each have degree at most $\Delta$, completing the proof in this case.

Assume now that $G$ contains a triangle $uvw$. Consider first the case that $\Delta<5n/9$. Then any two of the vertices $u,v,w$ have less than $t$ common neighbors, thus $d(u)+d(v)+d(w)<n+3t$. This implies that there is a vertex of degree less than $(n+3t)/3<13n/30$ here we use that $n$ is large enough). We delete such a vertex and denote the remaining graph by $G_1$. We repeat this procedure to obtain $G_2$ and so on, until we either reach a triangle-free graph $G_i$, or we reach $G_{\Delta-\lfloor n/2\rfloor}$. In the first case $G_i$ has maximum degree at least $\Delta-i$, thus $G$ contains at most $(\Delta-i)(n-\Delta)+13in/30$ edges, which completes the proof. In the second case $G_{\Delta-\lfloor n/2\rfloor}$ has at most $n^2/16+13n/30(\Delta-n/2)$ edges, which completes the proof. 

Consider now the case $\Delta>5n/9$. Let $v$ be a vertex of degree $\Delta$ and $U$ be the set of its neighbors. Let $W$ be the set of vertices with degree at least $n/2+t$, then $W$ is an independent set by the $B_t$-free property. This also implies that $U$ and $W$ are disjoint. Let $U'$ be the subset of $U$ consisting of vertices that have at least $(n-\Delta)/2+t$ neighbors in $V(G)\setminus U$, then $U'$ is an independent set. We have at most $t|U\setminus U'|$ edges inside $U$. On the other hand, for every vertex $u\in U\setminus U'$ we have at least $(n-\Delta)/2-t$ vertices outside $U$ not connected to $u$. 

Let us assume that there are $k$ vertices of degree less than $n/2+t$ outside $U$. Then the number of edges in $G$ is at most $k(n/2+t)+(n-\Delta-k)\Delta+t|U\setminus U'|$. Indeed, the first term stands for the edges incident to $V(G)\setminus U\setminus W$, the second term stands for the edges incident to $W$ and the third term stands for the edges inside $U$. The above quantity can be rewritten as $\Delta(n-\Delta)-k(\Delta-n/2-t)+t|U\setminus U'|\le \Delta(n-\Delta)-kn/18+t|U\setminus U'|$. It is easy to see that we are done if $k\ge 1$ and $|U\setminus U'|\le n/17t$, or if $k\ge 19t$. 

We also have $|E(G)|\le t|U\setminus U'|+k(n/2+t)+\Delta(n-\Delta)-|U\setminus U'|((n-\Delta)/2-t)$, as the first term stands for the edges inside $U$, the second term stands for the edges inside $V(G)\setminus U$, the third term stands for the number of pairs $(u,v)$ with $u\in U, v\in V(G)\setminus U$, and the fourth term stands for those such pairs that are not edges of $V(G)$. If $|U\setminus U'|> n/17t$ and $k<19t$, then this bound is $\Delta(n-\Delta)-cn^2$ and we are done. 
\end{proof}



We need a version for books of size increasing with $n$, when we also forbid another graph.

\begin{lemma}\label{lem2}
Let $1/2<\beta$, $\alpha<1/9$, $2\alpha <1-\beta$ and $F$ be a 3-chromatic graph with a color-critical edge. Let $n>n_0(\alpha,\beta,F)$ be an integer large enough, $t\le\alpha n$ and $n/2<\Delta<\beta n$ be positive integers. If $G$ is an $n$-vertex $F$-free graph with maximum degree $\Delta$ such that the largest book in $G$ is $B_t$, then $|E(G)|\le \Delta(n-\Delta)-ctn$ for a constant $c=c(\alpha,\beta,F)$.
\end{lemma}


\begin{proof}
We start by picking an $\varepsilon>0$ sufficiently small with respect to $\alpha$, $\beta$ and $F$. We will use that $\varepsilon$ is small in multiple instances; we will not list the exact assumptions.
We choose $n$ large enough with respect to $\alpha$, $\beta$, $F$ and $\varepsilon$; again we do not list the exact assumptions here.

Let $v$ be a vertex of degree $\Delta$ and $U$ be the set of its neighbors. Clearly the property that $H$ is $F$-free implies that a bipartite graph is forbidden inside $U$, thus there are at most $\varepsilon^4 n^2/3$ edges inside $U$ by the K\H ov\'ari-S\'os-Tur\'an theorem. Let $W$ be the set of vertices outside $U$ with degree at least $\Delta-\varepsilon n$. Assume that $|W|\le n-\Delta-\varepsilon^3 n/2$. Then the total number of edges in $G$ is at most $\Delta|W|+(\Delta-\varepsilon n)(n-\Delta-|W|)+\varepsilon^4 n^2/3=\Delta(n-\Delta)-\varepsilon n(n-\Delta-|W|)+\varepsilon^4 n^2/3\le \Delta(n-\Delta)-\varepsilon^4 n^2/6$, completing the proof. Thus we will assume that $|W|> n-\Delta-\varepsilon^3 n/2$. Two vertices of $W$ have at least $\alpha n$ common neighbors (since $\varepsilon$ is small enough), thus there are no edges inside $W$. This implies that all but at most $(\varepsilon^3/2+\varepsilon^4/2)n^2$ edges of $G$ go between $U$ and $W$.

Let $u,u'\in U$ with $uu'\in E(G)$. Then $u$ and $u'$ have at most $t\le \alpha n$ common neighbors, in particular the number of neighbors of $u$ in $W$ plus the number of neighbors of $u'$ in $W$ is at most $n-\Delta+\alpha n$. Let $U'$ denote the set of vertices in $U$ with more than $|W|-\varepsilon n\ge n-\Delta-2\varepsilon n>(n-\Delta+\alpha n)/2$
neighbors in $W$, then $U'$ is an independent set. Every vertex of $U\setminus U'$ is connected to at most $|W|-\varepsilon n$ vertices of $W$, thus at least $\varepsilon n|U\setminus U'|$ edges are missing between $U$ and $W$. 

If $|U\setminus U'|\ge \varepsilon^2 n$, then the total number of edges is at most $|U||W|-\varepsilon^3 n^2+(\varepsilon^3/2+\varepsilon^4/2)n^2= |U||W|-(\varepsilon^3-\varepsilon^4)n^2/2\le  \Delta(n-\Delta)-(\varepsilon^3-\varepsilon^4)n^2/2$, completing the proof. Thus we will assume that $|U\setminus U'|\le \varepsilon^2 n$.

We found a bipartite subgraph $G'$ of $G$ with parts $U'$ of size at least $\Delta-\varepsilon^2 n$ and $W$ of size at least $n-\Delta-\varepsilon^3n/2\ge n-\Delta-\varepsilon^2 n$. 
Moreover, every vertex of $G'$ is connected to all but at most $\varepsilon n$ vertices of the other part of $G'$.

Let $x$ be a vertex of $G$ not in $G'$. If $x$ has a set $A$ of $|V(F)|$ neighbors in $W$, then the vertices in $A$ have at least $|U'|-\varepsilon n|V(F)|$ common neighbors in $U'$. Because of the $|V(F)|$-free property, $x$ is connected to less than $|V(F)|$ of those neighbors, thus to at most $|V(F)|(\varepsilon n+1)$ vertices of $U'$. Similarly, if $x$ has at least $|V(F)|$ neighbors in $U'$, then $x$ has at most $|V(F)|(\varepsilon n+1)$ neighbors in $W$. Thus either $x$ is connected to at most $|V(F)|(\varepsilon n+1)$ vertices on both sides of $G'$, or $x$ is connected to less than $|V(F)|$ vertices on one side.

We denote by $U''$ the union of $U'$ and the the set of vertices outside $G'$ that are connected to less than $|V(F)|$ vertices of $U'$ and at least $t+1$ vertices of $W$, and similarly we denote by $W'$ the union of $W$ and the the set of vertices outside $G'$ that are connected to less than $|V(F)|$ vertices of $W$ and at least $t+1$ vertices of $U'$. Then the number of edges inside $U''$ is at most $2|V(F)|\varepsilon^2n$, 
as the at most $2\varepsilon^2n$ vertices added to the independent set $U'$ are each belong to less than $|V(F)|$ such edges. Similarly there are at most $2|V(F)|\varepsilon^2n$ edges inside $W'$.

\begin{clm}\label{newcl}
There is a $p_0$ such that if $p\ge p_0$ and we have a $B_p$ in $G$ with vertices $x,y,z_1,\dots,z_p$, edges $xy$, $xz_i$, $yz_i$ for $i\le p$ and $z_1,\dots,t_p\in U''\cup W'$, then at least $c_0pn$ edges are missing between $U''$ and $W'$ for some constant $c_0$ depending only on $F$ and $\beta$.
\end{clm}

\begin{proof}
 There is a $p_0(F)$ such that if $p\ge p_0(F)$, then $\ex(2p,K_{|V(F)|,|V(F)})\le \varepsilon p^2$ by the K\H ov\'ari-S\'os-Tur\'an theorem. At most $\varepsilon p^2$ edges go from $z_1,\dots,z_p$ to any set of $p$ vertices not containing $x$ and $y$, hence at most $\lceil n/p\rceil \varepsilon p^2+2p\le 2p\varepsilon n$ edges are incident to $z_1,\dots,z_p$. Thus at least $p(n-\Delta-\varepsilon^2 n-2\varepsilon n)/2\ge pn(1-\beta-\varepsilon^2-2\varepsilon)/2\ge pn(1-\beta)/3$ edges are missing between $U''$ and $W'$ if $\varepsilon$ is small enough.
\end{proof}

Assume first that each vertex of $V(G)\setminus V(G')$ is in $U''\cup W'$.
Observe that $U\subset U''$. Indeed, if a vertex $u\in U$ is connected to $t+1$ vertices of $U'$, then $uv\in E(G)$ and $u$ and $v$ have $t+1$ common neighbors, forming $B_{t+1}$, a contradiction. Therefore, $|U''|\ge \Delta$. 


If $t<p_0(F)$, the statement we want to prove simply follows from Lemma \ref{lem1}. If $t\ge p_0(F)$, then Claim \ref{newcl} with $p=t$ completes the proof, since the number of edges that are not between $U''$ and $W'$ is at most $4|V(F)|\varepsilon^2n\le c_0tn/2$.


Assume now that there are $1\le k\le 2\varepsilon^2n$ vertices not in $U''\cup W'$. They are either each connected to at most $t$ vertices of both $U'$ and $W$, or to less than $|V(F)|(\varepsilon n+1)$ vertices of both $U'$ or $W$. Observe that $\max\{|V(F)|(\varepsilon n+1),t\}\le (n-\Delta-\varepsilon n)/2$ for $\varepsilon$ small enough, thus these $k$ vertices are incident to at most $k(n-\Delta-\varepsilon n)$ edges.
Then $|U|\ge \Delta-k$, thus $|U''||W'|\le (\Delta-k)(n-\Delta)$.
The number of edges in $G$ is at most $(\Delta-k)(n-\Delta)+4|V(F)|\varepsilon^2n+k(n-\Delta-\varepsilon n)\le \Delta(n-\Delta)-k\varepsilon n/2$. This completes the proof if $t<2p_0(F)$ or if $k>t/2$. In the remaining case, at least $t/2$ vertices of $z_1,\dots,z_t$ are in $U''\cup W'$, thus we can apply Claim \ref{newcl} with $p=\lceil t/2\rceil$. Then at least $c_0\lceil t/2\rceil n$ edges are missing between $U''$ and $W'$, and the number of edges not between $U''$ and $W'$ is at most $4|V(F)|\varepsilon^2n+k|V(F)|(\varepsilon n+1)\le c_0tn/3$ if $\varepsilon$ is small enough.
\end{proof}

Let us continue with the lemmas needed for the proof of Theorem \ref{fketto}.
A \textit{vertex cover} is a set of vertices incident to each of the edges.

\begin{lemma}\label{vc}
If $G$ is an $n$-vertex $F_2$-free graph and $A$ is a vertex cover with $|A|<n/4$, then $|E(G)|<|A|(n-|A|)+1$.
\end{lemma}

\begin{proof}
Assume that $A$ is a vertex cover. Let us delete each vertex of $A$ that is connected to at most $n/2$ vertices of $V(G)\setminus A$, and denote the resulting set by $A'$.
Assume that there are $\ell>1$ edges inside $A'$. Then there is a matching or star in $A'$ with $\sqrt{\ell}$ edges by Theorem \ref{abhasa}. 

Observe that a vertex outside $A$ can be connected to at most $m+1$ vertices of a matching with $m$ edges inside $A'$, because of the $F_2$-free property. For a star inside $A'$ with $m$ edges, at most one vertex outside $A$ can be connected to its center and two other vertices. Therefore, $n-|A|-1$ vertices outside $A$ are each either connected to at most 2 vertices of the star, or are not connected to the center. The latter possibility occurs for less than $n/2$ vertices outside $A$, as the center of the star is in $A'$. This means that in both cases at least $(m-1)(n/2-|A|)$ edges are missing between $A'$ and $V(G)\setminus |A|$.

The number of edges inside $A$ is at most $\ell+|A||A\setminus A'|$, while the number of edges missing between $A$ and $V(G)\setminus |A|$ is at least $(\sqrt{\ell}-1)(n/2-|A|)+|A\setminus A'|(n/2-|A|)$. Observe that $\ell\le \binom{|A'|}{2}<|A'|^2/2$, thus $\sqrt{2\ell}<|A'|<n/2-|A|$, which implies that  $\ell=\sqrt{\ell/2}\sqrt{2\ell}<(\sqrt{\ell}-1)(n/2-|A|)$. We also have $|A||A\setminus A'|<|A\setminus A'|(n/2-|A|)$, and the two inequalities combined mean that more edges are missing between $A$ and $V(G)\setminus |A|$ than the number of edges inside $A$, which completes the proof.

Finally, if $\ell\le 1$, then we have at most $\ell+|A||A\setminus A'|$ edges inside $A$ and at least $|A\setminus A'|(n/2-|A|)\ge |A||A\setminus A'|$ edges are missing between $A$ and $V(G)\setminus A$, completing the proof.
\end{proof}

\begin{lemma}\label{deltaersim} Let $a$ and $b$ be positive integers.
For any $\varepsilon>0$ there exists $\delta>0$ such that if $G$ is an $n$-vertex $F_2$-free graph with maximum degree $\Delta>n/2$ and at least $\Delta(n-\Delta)-\delta n^2$ edges, then $G$ can be turned to a bipartite graph $G'$ by deleting at most $\varepsilon n^2$ edges. 
\end{lemma}

We remark that if $\Delta=n/2+o(n)$, then this is the Erd\H os-Simonovits stability.

\begin{proof} By the Erd\H os-Simonovits stability theorem, there is an $\alpha>0$ such that if $G$ has at least $n^2/4-\alpha n^2$ edges, then $G$ can be made bipartite by deleting at most $\varepsilon n^2/2$ edges. We can assume that $\alpha$ is small enough. Assume first that $\Delta\le n/2+\alpha n$, and let $\delta=\alpha/2$. Then $G$ has at least $\Delta(n-\Delta)-\delta n^2\ge n^2/4-(\alpha^2+\delta)n^2\ge n^2/4-\alpha n^2$ edges, thus the conclusion holds.

Let us assume now that $\Delta>n/2+\alpha n$ and let $v$ be a vertex of degree $\Delta$ and $U$ be the set of its neighbors. Then there are no independent edges inside $U$, hence there are at most $n$ edges inside $U$. Consider the bipartition of $G$ into $U$ and $V(G)\setminus U$. Let $G'$ be the subgraph of $G$ induced on $V(G)\setminus U$. We are done unless $|E(G')|\ge \varepsilon n^2-n$. Let $u\in U$ and let $W$ be the set of vertices in $V(G)\setminus U$ that are not connected to $u$. Then $W$ covers almost all the edges of $G'$. More precisely, there is a vertex $w$ such that $W\cup \{w\}$ is a vertex cover of $G'$. Indeed, there can be only a star outside $W$ in $G'$, as otherwise there are two independent edges in the neighborhood of $u$.

We pick $u\in U$ such a way that $W$ is the smallest, then at least $\Delta|W|$ edges are missing between $U$ and $V(G)\setminus U$. By Lemma \ref{vc}, if $|W|+1<n/4$, then $|E(G')|<(|W|+1)(n-\Delta-|W|-1)+1$. If $|W|<\varepsilon n$, 
this contradicts our assumption that $|E(G')|\ge \varepsilon n^2-n$. Therefore, we can assume that $|W|\ge\epsilon n$. 
 The number of edges in $G$ is at most $\Delta(n-\Delta)-\Delta|W|+(|W|+1)(n-\Delta-|W|-1)+1<\Delta(n-\Delta)-\delta n^2$ if $\delta$ is small enough, a contradiction.

If $|W|+1\ge n/4$, then we use the simpler bound $|E(G')|\le \binom{|V(G')|}{2}$, thus the number of edges in $G$ is at most $\Delta(n-\Delta)-\Delta|W|+\binom{n/2-\alpha n}{2}<\Delta(n-\Delta)-\delta n^2$ if $\delta$ is small enough, a contradiction completing the proof.
\end{proof}

\section{Proofs}
Recall that Theorem \ref{cce} states that if $F$ has chromatic number three and a color-critical edge, then $S_{a,b}$ is weakly $F$-Tur\'an-good.
Before the proof, we describe the main component of the proof of Theorem \ref{GyWW} in \cite{gyww}. We let $f(x,y)=\binom{x-1}{a}\binom{y-1}{b}+\binom{y-1}{a}\binom{x-1}{b}$ if $a\neq b$ and $f(x,y)=\binom{x-1}{a}\binom{y-1}{b}$ if $a=b$.
Then $f(d(u),d(v))$ is the number of copies of $S_{a,b}$ with central edge $uv$. Gy\H ori, Wang and Woolfson observed that if $G$ is a triangle-free graph with maximum degree $\Delta\ge n/2$, then $f(d(u),d(v))\le \max_{n/2\le \Delta'\le \Delta} f(\Delta',n-\Delta')$. It is easy to see that $K_{\Delta',n-\Delta'}$ has at least as many edges as $G$, and each edge is the central edge of as least as many copies of $S_{a,b}$ as any edge of $G$. 

\begin{proof}[Proof of Theorem \ref{cce}]
Let us assume indirectly that there is an infinite sequence of counterexamples, i.e. graphs $G_1,G_2,\dots$ such that $\ex(n_i,S_{a,b},F)=\cN(S_{a,b},G_i)>\cN(S_{a,b},K_{m,n_i-m})$ for every $m$, and $n_i>n_{i-1}$. 
Let us consider first the case that there exists a constant $\beta<1$ such that for each $G_i$, the maximum degree $\Delta_i<\beta n_i$. 

First we show that $\Delta_i>n_i/2$. Indeed, otherwise $K_{\lfloor n_i/2\rfloor,\lceil n_i/2\rceil}$ has at least $|E(G)|$ edges, and each edge of $K_{\lfloor n_i/2\rfloor,\lceil n_i/2\rceil}$ is the central edge of $\binom{\lfloor n_i/2\rfloor-1}{a}\binom{\lceil n_i/2\rceil-1}{b}+\binom{\lceil n_i/2\rceil-1}{a}\binom{\lfloor n_i/2\rfloor-1}{b}$ copies of $S_{a,b}$. On the other hand, in $G_i$ each edge is the central edge of at most $\binom{\Delta_i-1}{a}\binom{\Delta_i-1}{b}+\binom{\Delta_i-1}{a}\binom{\Delta_i-1}{b}$ copies of $S_{a,b}$. This implies that $\cN(S_{a,b},G_i)\le \cN(S_{a,b},K_{\lfloor n_i/2\rfloor,\lceil n_i/2\rceil})$, a contradiction.

Thus, we can assume that $\Delta_i>n_i/2$. Let us pick $\varepsilon>0$. By a result of Alon and Shikhelman \cite{ALS2016} mentioned earlier, there exists an $\alpha>0$ such that any $B_{\alpha n}$-free graph has at most $\varepsilon n^2$ triangles, if $n$ is large enough. We choose 
an $i$ such that $n_i$ is large enough. Let $G=G_i$, $n=n_i$ and $\Delta=\Delta_i$. Let $B_t$ be the largest book in $G$ with vertices $u,v$ connected to each other and to $z_1,\dots,z_t$.
We consider first the case when $t\ge\alpha n$. Then there is no copy of $K_{|V(F)|,|V(F)|}$ with one part in $Z=\{z_1,\dots,z_t\}$. Indeed, such a copy together with $u$ and $v$ would contain $F$. This implies with the K\H ov\'ari-S\'os-Tur\'an theorem that there are $o(n^2)$ edges incident to $Z$ in $G$. Let us consider the copies of $S_{a,b}$ that contain a vertex from $Z$. We can count them by picking an edge incident to $Z$ $o(n^2)$ ways, then $a+b$ other vertices $O(n^{a+b})$ ways, and then a copy of $S_{a,b}$ on these $a+b+2$ vertices, constant number of ways. This shows that there are $o(n^{a+b+2})$ copies of $S_{a,b}$ that contain a vertex from $Z$. Let us delete all the edges incident to $Z$, and add a copy of $K_{\lfloor t/2\rfloor,\lceil t/2\rceil}$ on $Z$. The resulting graph is clearly $F$-free and contains $\cN(S_{a,b},G)-o(n^{a+b+2})+\Omega(n^{a+b+2})>\cN(S_{a,b},G)$ copies of $S_{a,b}$, a contradiction.

Assume now that  $t<\alpha n$. If $t=0$, then Theorem \ref{GyWW} gives the proof. If $t>0$, then we can apply Lemma \ref{lem2} to show that $G$ has at most $\Delta(n-\Delta)-ctn$ edges. Observe that for every edge $uv$ we have that $d(u)+d(v)\le n+t$. 

Consider the maximum of the function $f(x,y)$ where $x+y\le z$ and $x,y\le \Delta$. Clearly this maximum is obtained by $f(\Delta',z-\Delta')$ for some $z/2\le \Delta'\le \Delta$. We apply it for $z=n$, thus for $x,y\le n$ we have $f(x,y)\le f(m,n-m)$ for some $n/2\le m\le \Delta$. If $x',y'\le n+g$ for some $g$, then we have that $f(x',y')\le f(m',n+g-m')\le f(m',n-m')+c'gn^{a+b-1}\le f(m,n-m)+c'gn^{a+b-1}$ for some $n/2\le m'\le \Delta$ and a constant $c'$ depending only on $a$ and $b$.
Observe that each edge of $K_{m,n-m}$ is the central edge of $f(m,n-m)$ copies of $S_{a,b}$, and $K_{m,n-m}$ has $m(n-m)\ge \Delta(n-\Delta)$ edges. 

For an edge $uv$, let $g(uv)$ denote the number of common neighbors of $u$ and $v$. Clearly $\sum_{uv\in E(G)}g(uv)\le 3\cN(K_3,G)\le 3\varepsilon n^2$. We have $d(u)+d(v)\le n+g(uv)$, thus $f(d(u),d(v))\le f(\Delta'',n+g(uv)-\Delta'')$ for some $\Delta''\le \Delta$. We have $f(\Delta'',n+g(uv)-\Delta'')\le f(\Delta'',n-\Delta'')+g(uv)n^{a+b-1}\le f(m,n-m)+g(uv)n^{a+b-1}$. Therefore, $\cN(S_{a,b},G)=\sum_{uv\in E(G)}f(d(u),d(v))\le |E(G)|f(m,n-m)+\sum_{uv\in E(G)}g(uv)n^{a+b-1}\le |E(G)|f(m,n-m)+3\varepsilon n^{a+b+1}\le (\Delta(n-\Delta)-ctn)f(m,n-m)+3\varepsilon n^{a+b+1}\le m(n-m)f(m,n-m)-ctnf(m,n-m)+3\varepsilon n^{a+b+1}\le m(n-m)f(m,n-m)-c''tn^{a+b+1}+3\varepsilon n^{a+b+1}$ for some constant $c"'$ depending on $\alpha,\beta,F$. This completes the proof if $t\ge 3\varepsilon/c''$. 

For $t<3\varepsilon/c''$, we use Lemma \ref{lem1}. It shows that $G$ has at most $\Delta(n-\Delta)-c_1tn$ edges, where $c_1$ does not depend on $\alpha$, hence it does not depend on $\varepsilon$. The same calculation as above gives $\cN(S_{a,b},G)\le m(n-m)f(m,n-m)-c_1'tn^{a+b+1}+3\varepsilon n^{a+b+1}$ for some $c_1'$ depending on $\beta$ and $F$. This completes the proof as $\varepsilon$ is small enough.

Finally, if $\beta$ does not exist, then we pick a $\gamma>0$ small enough and we can pick an $n$-vertex $F$-free graph $G$ with maximum degree $\Delta>(1-\gamma)n$ and $\cN(S_{a,b},G)=\ex(n,S_{a,b},F)$ such that $n$ is large enough. Let $v$ be a vertex of maximum degree $\Delta$ and $U$ be its neighborhood. Observe that if we delete the endvertex of a color-critical edge of $F$, we obtain a bipartite
graph $F''$. Then there is no $F''$ on $U$, thus by the K\H ov\'ari-S\'os-Tur\'an theorem, there are at most $\gamma n^2$ edges inside $U$ if $n$ is large enough. As there are at most $\gamma n$ vertices outside $U$, there are at most $\gamma n^2$ edges of $G$ incident to a vertex outside $U$, thus $|E(G)|\le 2\gamma n^2$. This implies that $\cN(S_{a,b},G)\le 2\gamma n^{a+b+2}$, as we can pick the central edge first, and then we can pick each other vertex at most $n$ ways. Clearly $\cN(S_{a,b},K_{\lfloor n/2\rfloor,\lceil n/2\rceil})> 2\gamma n^{a+b+2}$ if $\gamma$ is small enough and $n$ is large enough, giving us the contradiction that completes the proof.
\end{proof}


Recall that Theorem \ref{fketto} states that for sufficiently large $n$ we have $\ex(n,S_{a,b},F_2)=\cN(S_{a,b},K_{m,n-m}^+)$ for some $m$.
The proof is similar to that of Theorem \ref{cce}. In particular, we use $f(x,y)$ and $g(uv)$. 

\begin{proof}[Proof of Theorem \ref{fketto}]
Similarly to Theorem \ref{cce}, we assume indirectly that there is an infinite sequence of counterexamples, i.e. graphs $G_1,G_2,\dots$ such that $|V(G_i)|=n_i$, $\ex(n_i,S_{a,b},F_2)=\cN(S_{a,b},G_i)>\cN(S_{a,b},K_{m,n_i-m})$ for every $m$, and $n_i>n_{i-1}$. 
Let us consider first the case that there exists a constant $\beta<1$ such that for each $G_i$, the maximum degree $\Delta_i<\beta n_i$. 

First observe that $\Delta_i>n_i/2$ by the same reasoning as in Theorem \ref{cce}. Let $\varepsilon>0$ be small enough and $n=n_i$ be large enough. Let $\Delta=\Delta_i$. We will also prove that $m,n-m\le \Delta$.
It is easy to see and follows from a more general theorem of Alon and Shikhelman \cite{ALS2016} that $\ex(n,K_3,F_2)\le c_1n$ for an absolute constant $c_1$. 


Assume first that each vertex has degree at least $3\varepsilon^{1/3} n$.

\begin{clm}
$G$ has at most $(\Delta-1)(n-\Delta+1)-c_2n$ edges for some $c_2$ that depends of $\varepsilon$ but does not depend on $n$.
\end{clm}

\begin{proof}[Proof of Claim]
Assume that the statement does not hold, then by Lemma \ref{deltaersim} there is a bipartite subgraph $G'$ of $G$ with parts $A$ and $B$ such that there are at most $\varepsilon n^2$ edges inside $A$ and $B$ in $G$. We pick among such bipartite graphs the one with the most edges. It implies that each vertex is connected to at least as many vertices in the other part as in its part, in particular to at least $3\varepsilon^{1/3} n/2$ vertices. 

Let $u$ be a vertex of degree $\Delta$, and assume that $u$ is connected to $a$ vertices in part $A$, where $|A|\ge |B|$. Then there are at most $\Delta/2\le n$ edges between its neighbors, thus at least $a(\Delta-a)-2n$ edges are missing between $A$ and $B$. Therefore, the number of edges in $G$ is at most $|A|(n-|A|)+\varepsilon n^2-a(\Delta-a)+2n \le |A|(n-|A|)-|A|(\Delta-|A|)+\varepsilon n^2+2n = |A|(n-\Delta)+\varepsilon n^2+2n$, using that $\varepsilon$ is small enough and $a\le |A|$, thus $a(\Delta-a)\ge |A|(\Delta-|A|)$. Furthermore, if $|A|-a\ge \varepsilon^{1/3} n$, then $a(\Delta-a)\ge |A|(\Delta-|A|)+(2|A|-\Delta)\varepsilon^{1/3}n-\varepsilon^{2/3} n^2\ge |A|(\Delta-|A|)+ \varepsilon^{1/3}(1-\beta-\varepsilon^{1/3})n^2$. If $\varepsilon$ is small enough, this is larger than $|A|(\Delta-|A|)+\varepsilon n^2+3n$, completing the proof. Hence we can assume that $|A|-a<\varepsilon^{1/3} n$. If $B$ contains at least 2 neighbors of $u$, at least one of them has at most one common neighbor with $u$ in part $A$, in particular has less than $\varepsilon^{1/3} n$ neighbors in $A$, a contradiction. Therefore, $u$ has at most one neighbor in $B$, thus $a\ge \Delta-1$, in particular $|A|\ge \Delta-1$.

Assume first that there are $q\ge 2$ edges inside $A$ and $q'\le q$ edges inside $B$ in $G$ and recall that $q+q'\le \varepsilon n^2$ (the case $q'>q$ works the same way, we omit the details). Then by Theorem \ref{abhasa}, there is a star or matching of size at least $\sqrt{q}$ inside $A$. If there is a matching $M$ of size $m\ge\sqrt{q}/100$ inside $A$, then every vertex of $B$ is connected to at most one endpoint of all but one of the edges of $M$, thus to at most $m+1$ vertices of $M$. Therefore, at least $(n-\Delta)(m-1)$ edges are missing between $A$ and $B$. 

Assume that there is a star $S$ of size $m\ge\sqrt{q}$ inside $A$ with center $u$. Recall that at least $\varepsilon^{1/3} n$ vertices in $B$ are connected to $u$.
If $v\in B$ is connected to $u$ and a leaf $w$ of $S$, then no other vertex can be connected to both $u$ and another leaf of $S$. This means that at least $\varepsilon^{1/3} n-1$ vertices of $B$ are connected only to $u$ and $w$ in $S$, thus there are at least $(\varepsilon^{1/3} n-1)(m-1)$ missing edges between $A$ and $B$.

In both cases, at least $(\varepsilon^{1/3} n/2)\sqrt{q}$ edges are missing between $A$ and $B$. On the other hand, there are at most $2q= \sqrt{q}2\sqrt{q}\le \sqrt{q}2\sqrt{\varepsilon} n\le (\varepsilon^{1/3} n/4)\sqrt{q}$ edges inside $A$ and $B$. Therefore, the number of edges in $G_0$ is at most $|A||B|-(\varepsilon^{1/3} n/4)\sqrt{q}\le (\Delta-1)(n-\Delta+1)-(\varepsilon^{1/3} n/4)\sqrt{q}$. Here we used that $A$ has order at least $\Delta-1$. 
\end{proof}

Let us return to the proof of the theorem.
We have  $\sum_{uv\in E(G)}g(uv)\le 3\cN(K_3,G)\le 3c_1 n$. Then for some $m\le \Delta-1$, we have $\cN(S_{a,b},G)=\sum_{uv\in E(G)}f(d(u),d(v))\le |E(G)|f(m,n-m)+\sum_{uv\in E(G)}g(uv)n^{a+b-1}\le |E(G)|f(m,n-m)+3c_1 n^{a+b}\le (\Delta-1)(n-\Delta+1)f(m,n-m)-c_2nf(m,n-m)+3c_1 n^{a+b}\le m(n-m)f(m,n-m)-c_2nf(m,n-m)+3c_1 n^{a+b}\le m(n-m)f(m,n-m)-c'n^{a+b+1}+3c_1 n^{a+b}$ for some constant $c'$ depending on $\alpha,\beta,F,\varepsilon$. 
This completes the proof.

Assume next that there is exactly one edge $uv$ inside $A$ and exactly one edge $xy$ inside $B$.
If no vertex of $A$ is connected to both $x$ and $y$, or no vertex of $B$ is connected to both $u$ and $v$, then $|E(G)|\le \Delta(n-\Delta)-\Omega(n)$, and the above calculation extends to this case to complete the proof. 

Otherwise, it is easy to see that two independent edges on these four vertices, say $ux$ and $vy$ are missing from $G$. We can assume without loss of generality that $uy$ and $vx$ are in $G$. Let $G_0$ denote $K_{|B|,|A|}^+$ with parts $A$ and $B$ and the edge $xy$. Let us compare the number of copies of $S_{a,b}$ in $G$ and in $G_0$. These graphs share the edges with an endvertex not in $\{u,v,x,y\}$, and the endvertices either have the same degree or the degree is larger in $G_0$. Moreover, the intersection of the neighborhoods of the endvertices is as small as possible, thus each such edge is the central edge of the same number of copies of $S_{a,b}$ in both graphs, or more copies of $S_{a,b}$ in $G_0$. The same holds for the edges $ux$ and $vy$. The edge $uv$ in $G$ is the central edge of $\binom{|B|}{a}\binom{|B|-a}{b}$
copies of $S_{a,b}$ in $G$, and the edge $ux$ is the central edge of more than $\binom{|B|-1}{a}\binom{n-|B|-1}{b}+\binom{|B|-1}{b}\binom{n-|B|-1}{a}$  
copies of $S_{a,b}$ in $G_0$ (in the case $a=b$, both quantities are divided by 2). Therefore, $G_0$ contains more copies of $S_{a,b}$ than $G$, a contradiction. 

Let us assume now that there are vertices of degree less than $3\varepsilon^{1/3} n$.
We remove vertices of degree less than $3\varepsilon^{1/3} n$ one by one, as long as there is such a vertex. Let $k$ be the number of vertices removed this way and $G_1$ be the resulting graph.
Each time we remove a vertex, we remove at most $3\varepsilon^{1/3} n^{a+b+1}$ copies of $S_{a,b}$, thus we remove at most $3k\varepsilon^{1/3} n^{a+b+1}$ copies of $S_{a,b}$.
If $k\ge \varepsilon^{1/4(a+b+1)} n$, then we place $K_{\lfloor k/2\rfloor,\lceil k/2\rceil}$ on those vertices, that has at least $c'k^{a+b+2}\ge c'k\varepsilon^{1/4}n^{a+b+1}$ copies of $S_{a,b}$ for some constant $c'$ that depends only on $a$ and $b$. Since $\varepsilon$ is small enough, the number of copies of $S_{a,b}$ increases without creating $F_2$, a contradiction. If $k<\varepsilon^{1/4(a+b+1)} n$, then $n-k$ is large enough, thus we know that $G_1$ contains at most $\cN(S_{a,b},K_{m,n-k-m}^+)$ copies of $S_{a,b}$ for some $m$ with $m,n-k-m\le \Delta$. Then $K_{m,n-m}^+$ contains more copies of $S_{a,b}$ than $G$. Indeed, we pick $k$ vertices in the part of order $n-m$, there are $\cN(S_{a,b},K_{m,n-k-m}^+)$ copies of $S_{a,b}$ that avoid them, and there are at least $k\binom{m}{a+1}\binom{n-m-k}{b}\ge c''k(1-\beta-\varepsilon^{1/4(a+b+1)})^{a+b+1}n^{a+b+1}$ copies of $S_{a,b}$ that contain exactly one of them, for some constant $c''$ that depends only on $a$ and $b$. If $\varepsilon$ is small enough, then $c''k(1-\beta-\varepsilon^{1/4(a+b+1)})^{a+b+1}n^{a+b+1}>3k\varepsilon^{1/3} n^{a+b+1}$, giving a contradiction.

Finally, consider the case that there is no $\beta<1$ such that for each $i$, $\Delta_i<\beta n_i$. Then we can proceed exactly as in the proof of Theorem \ref{cce}, thus we only give a sketch. We can pick a $\gamma$ small enough and an $F_2$-free graph $n$-vertex graph $G$ with $\cN(S_{a,b},G)=\ex(n,S_{a,b},F_2)$, maximum degree at least $(1-\gamma)n$ such that $n$ is large enough. Let $U$ be the neighborhood of a vertex of maximum degree, then there are at most $n$ edges inside $U$ and at most $\gamma n^2$ other edges, thus at most $2\gamma n^{a+b+2}$ copies of $S_{a,b}$, which is less than $\cN(S_{a,b},K_{\lfloor n/2\rfloor,\lceil n/2\rceil})$, a contradiction completing the proof.
\end{proof}


Let us continue with Proposition \ref{klikks}. Let $S_r$ denote the star with $k$ leaves, i.e. $S_r=S_{r,0}$, and let
$n=p(a+b+1)+q$ with $q\le a+b$. Chase \cite{chase}, proving a conjecture of Gan, Loh and Sudakov \cite{gls} showed that for $k\ge 3$ we have $\ex(n,K_k,S_{a+b+1})=p\binom{a+b+1}{k}+\binom{q}{k}$. Recall that Proposition \ref{klikks} states that $\ex(n,K_k,S_{a,b})=p\binom{a+b+1}{k}+\binom{q}{k}$.

\begin{proof}[Proof of Proposition \ref{klikks}]  Let $G=pK_{a+b+1}+K_q$ denote the vertex-disjoint union of $p$ copies of $K_{a+b+1}$ and a copy of $K_q$. Then $G$ is clearly $S_{a,b}$-free (and $S_{a+b+1}$-free) and contains $p\binom{a+b+1}{k}+\binom{q}{k}$ copies of $K_k$. This gives the lower bound.

Let us assume that $a\le b$ without loss of generality. We use induction on $n$ to prove the upper bound.
Let $G$ be an $S_{a,b}$-free $n$-vertex graph with the most copies of $K_k$. If every vertex of $G$ has degree at most $a+b$, then we are done by the result of Chase. If a vertex has degree $d>a+b$, then all its neighbors have degree less than $a$. This way we found $a+b+1$ vertices with degree less than $a$, let $G'$ be the graph obtained by deleting them. We deleted at most $(a+b+1)\binom{a-1}{k-1}$ copies of $K_k$. Let $G''$ be the $n$-vertex $S_{a,b}$-free graph obtained by adding a $K_{a+b+1}$ to $G'$. Then $G''$ contains at least $\cN(K_k,G)-(a+b+1)\binom{a-1}{k-1}+\binom{a+b+1}{k}$ copies of $K_k$. We have $(a+b+1)\binom{a-1}{k-1}=k(a+b+1)(a-1)\dots (a-k+1)/k!<k(a+b+1)\frac{a+b}{2}\dots\frac{a+b-k+2}{2}/k!=k\binom{a+b+1}{k}/2^{k-1}<\binom{a+b+1}{k}$, thus $G''$ contains more copies of $K_k$ than $G$, a contradiction.
\end{proof}

Let us turn to $\ex(n,S_{a,b},S_{c,d})$. Recall that Proposition \ref{neww} deals with the case $c\le a\le b<d$.

\begin{proof}[Proof of Proposition \ref{neww}]
In the calculations in this proof, we assume that $a\neq b$. If $a=b$, then our bounds on the number of copies of $S_{a,b}$ are always divided by 2 because of symmetry, thus the proof also works for that case. 

Let $G$ be an $n$-vertex $S_{c,d}$-free graph.
Let us assume there is a vertex $u$ of degree more than $c+d$ in $G$. Then each neighbor $v$ of $u$ has degree at most $c$, hence cannot be a center of a copy of $S_{a,b}$. As each vertex of $S_{a,b}$ is incident to a center, $u$ cannot be in any copy of $S_{a,b}$, thus we can delete $u$. Therefore, we can assume that each vertex of $G$ has degree at most $c+d$.

Assume now that a vertex $u$ has degree more than $d$. Then for any neighbor $v$ of $u$, there are at most $c+d+1$  vertices (including $u$ and $v$) that are connected to either $u$ or $v$. Therefore, there are at most $s_1=2\binom{c+d-1}{a}\binom{c+d-1-a}{b}$ copies of $S_{a,b}$ with central edge $uv$.
Moreover, if $v$ has degree at most $d$, then there are at most $s_2=\binom{d-1}{a}\binom{c+d-1-a}{b}+\binom{d-1}{b}\binom{c+d-1-b}{a}$ copies of $S_{a,b}$ with central edge $uv$. If $u$ and $v$ both have degree at most $d$, then clearly there are at most $s_3=2\binom{d}{a}\binom{d}{b}$ copies of $S_{a,b}$ with central edge $uv$. Observe that $s_2< s_3$ and $s_2<s_1$.

Let us assume that $G$ contains $m$ vertices of degree at most $d$, and there are $x$ edges between vertices of degree at most $d$ and more than $d$. Then there are at most $(dm-x)/2$ edges between vertices of degree at most $d$ and at most $((c+d)(n-m)-x)/2$ edges between vertices of degree more than $d$. This implies that the total number of copies of $S_{a,b}$ is at most $((c+d)(n-m)-x)s_1/2+s_2x+(dm-x)s_3/2\le (c+d)(n-m)s_1/2+dms_3/2$. This upper bound is linear in $m$, thus the maximum is taken at either $m=0$ or $m=|E(G)|$, and we have $|E(G)|\le (c+d)n/2$. This completes the proof of the upper bound.

If $r$ is nice, then the lower bounds are given by any $d$-regular triangle-free graph on $n$ or $n-1$ vertices. If $r$ is not nice, then the lower bounds are given by $\lfloor n/(c+d+1)\rfloor$ vertex disjoint copies of $K_{c+d+1}$.
\end{proof}


Recall that Theorem \ref{cnc} states that if $a\le b$, $a<c\le d$ and $n$ is sufficiently large, then $\ex(n,S_{a,b},S_{c,d})=\cN(S_{a,b},K_{c,n-c})$.

\begin{proof}[Proof of Theorem \ref{cnc}]
Let $G$ be an $S_{c,d}$-free $n$-vertex graph. Let $A$ denote the set of vertices with degree at most $c$ and $B$ denote the set of vertices of degree at least $c+d+1$. Then every edge from $B$ goes to $A$.
Let $q=\max\{c,|B|\}$ and $x=\binom{c-1}{a}\binom{|A|-a-1}{b}+\binom{c-1}{b}\binom{|A|-b-1}{a}$ for $a\neq b$ and $x=\left(\binom{c-1}{a}\binom{|A|-a-1}{b}+\binom{c-1}{b}\binom{|A|-b-1}{a}\right)/2$ for $a=b$. 

We consider two types of copies of $S_{a,b}$. There are $O(n)$ copies where the central edge is not incident to $B$, since there are $O(n)$ edges in an $S_{c,d}$-free graph and each edge not between $A$ and $B$ has endvertices of degree at most $c+d$.
The number of copies of $S_{a,b}$ with central edge between $A$ and $B$ is at most $|A|qx$. Indeed, there are at most $|A|q$ such edges, we pick $a$ neighbors of the endvertex in $A$, and the other endvertex has at most $|A|-a-1$ other neighbors, we pick $b$ of those vertices, and then repeat this with the role of $a$ and $b$ reversed, if $a\neq b$. Note that $|A|-a-1$ can be replaced here by the degree of the endvertex in $B$ to obtain a better bound.

We are done if $q<c$. Therefore, $B$ has at least $c$ vertices, let $v_1,\dots,v_c$ be the $c$ vertices of largest degree (we choose arbitrarily in the case of ties). If $v_c$ has degree $O(1)$, then we are also done, as only $|A|(c-1)$ edges between $A$ and $B$ would be central edges of $x$ copies of $S_{a,b}$, the other edges between $A$ and $B$ are central edges of $O(1)$ copies.

If a vertex $u$ has degree less than $c$, we can add an edge $uv_i$ without creating $S_{c,d}$. We repeat this until every vertex has degree at least $c$.
Assume next that there is an edge $uv$ between two vertices of degree $c$. Then we delete this edge and we can replace it with $uv_i$ and $vv_j$ for some $i$ and $j$. Clearly the number of copies of $S_{a,b}$ increases. We repeat this until we eliminate every such edge.

Assume now that a vertex $u$ of degree $c$ is connected to a vertex $w$ not among the vertices $v_i$. Then we replace the edge $uw$ with an edge $uv_i$ for some $i$. We claim that the number of copies of $S_{a,b}$ does not decrease. If $w$ has degree at most $c+d$, then we delete $O(1)$ copies of $S_{a,b}$ and add more. Otherwise every neighbor of $w$ has degree $c$. The number of copies of $S_{a,b}$ where $uw$ is the central edge depends only on the degrees of $u$ and $w$, thus there are more copies with $uv_i$ as the central edge. If $uw$ is a leaf edge, the number of copies depends on the degree of $w$ and of the other endvertex $w'$ of the central edge, but that is always $c$, and the same holds for $uv_i$ after the replacement. Here we use that there is no edge between vertices of degree $c$, thus after picking $a-1$ or $b-1$ other neighbors of $w$ or $v_i$, the neighbors of $w'$ are not among the vertices picked earlier. 

Assume that there is a vertex $u$ of degree more than $c$ but not in $\{v_1,\dots, v_c\}$. Then the neighbors of $u$ have degree more than $c$ and at most $c+d$, thus $u$ has degree at most $c+d$, and the neighbors of neighbors of $u$ also have degree at most $c+d$ and more than $c$. This implies that $u$ is in $O(1)$ copies of $S_{a,b}$. Let us delete the edges incident to $u$ and add the edges $uv_i$ for every $i\le c$, then the number of copies of $S_{a,b}$ increases. We repeat this until every vertex but $v_1,\dots,v_c$ has degree $c$. The resulting graph is $K_{c,n-c}$, completing the proof.
\end{proof}

\bigskip

\textbf{Funding}: Research supported by the National Research, Development and Innovation Office - NKFIH under the grants KH 130371, SNN 129364, FK 132060, and KKP-133819.

\end{document}